\documentclass{article}

\usepackage{multirow}

\usepackage[margin=1in]{geometry}
\usepackage{amsmath,amsthm,amssymb}
\usepackage{epsfig,epstopdf}
\usepackage{graphicx}
\usepackage{color}
\usepackage[dvipsnames]{xcolor}
\usepackage{subcaption}
\usepackage[export]{adjustbox}
\usepackage{url}

\newtheorem{theorem}{Theorem}

\newtheorem{lemma}{Lemma}

\graphicspath{ {./nlmc/} }

\begin{document}

\title{Space-Time Nonlinear Upscaling Framework Using Non-local Multi-continuum Approach}
\author{
Wing T. Leung\thanks{ICES, University of Texas, Austin, TX, USA (\texttt{wleung@ices.utexas.edu})}
\and
 Eric T. Chung\thanks{Department of Mathematics, The Chinese University of Hong Kong, Shatin, New Territories, Hong Kong SAR, China (\texttt{tschung@math.cuhk.edu.hk}) }
\and
Yalchin Efendiev\thanks{Department of Mathematics \& Institute for Scientific Computation (ISC), Texas A\&M University,
College Station, Texas, USA (\texttt{efendiev@math.tamu.edu})}
\and
Maria Vasilyeva\thanks{Institute for Scientific Computation (ISC), Texas A\&M University, College Station, Texas, USA (\texttt{efendiev@math.tamu.edu})}
\and
 Mary Wheeler\thanks{ICES, University of Texas, Austin, TX, USA (\texttt{mfw@ices.utexas.edu})}
}

\maketitle

\begin{abstract}

In this paper, we develop a space-time upscaling framework that can
be used for many challenging porous media applications without
scale separation and high contrast. Our main focus is on
nonlinear differential equations with multiscale coefficients.
The framework is built on nonlinear nonlocal multi-continuum
upscaling concept \cite{NLMC} and significantly extends
the results in the proceeding paper \cite{chung2018nonlinear}.

Our approach starts with a coarse space-time partition and
identifies test functions for each partition, which play a role
of multi-continua. The test functions are defined via optimization
and play a crucial role in nonlinear upscaling. In the second stage,
we solve nonlinear
local problems in oversampled regions with some constraints
defined via test functions. These local solutions define a nonlinear
map from macroscopic variables determined with the help of test functions
to the fine-grid fields. This map can be thought as a downscaled map
from macroscopic variables to the fine-grid solution. In the final
stage, we seek macroscopic variables in the entire domain
 such that the downscaled field
solves the global problem in a weak sense defined using the test functions.
We present an analysis of our approach for an example nonlinear problem.

Our unified framework plays an important role
in designing various upscaled methods.
Because local problems are directly related to the fine-grid
problems, it simplifies the process of finding local solutions with
appropriate constraints \cite{NLMC}. Using machine learning (ML),
we identify the complex map from macroscopic variables to fine-grid
solution.
 We present numerical results for several porous media
applications, including two-phase flow and transport.

\end{abstract}

\section{Introduction}

Many porous media models are nonlinear and deriving these
nonlinear macroscopic equations rely on some assumptions. For example,
the well-known two-phase flow and transport model assumes that the
relative permeabilities are functions of local saturations \cite{BT}.
Similarly, for unsaturated flows, the nonlinear relations between
pressures and capillary curves use local relations. All these problems
have space-time heterogeneities. Some rigorous
upscaling tools are needed to generalize these models and understand
the errors associated in these macroscopic models. This is one of
our goals in this paper.

Many approaches are suggested for nonlinear upscaling in the past,
e.g., \cite{ab05, egw10,  arbogast02, GMsFEM13, AdaptiveGMsFEM, brown2014multiscale, ElasticGMsFEM, ee03, abdul_yun, ohl12, fish2004space, fish2008mathematical, oz07, matache2002two, apwy07, henning2009heterogeneous, OnlineStokes, chung2017DGstokes,WaveGMsFEM, pwy02, Arbogast_PWY_07, MsDG, fish1997computational,oskay2007eigendeformation,yuan2009multiple}.
For multi-phase flows, these techniques include
permeability or transmissibility
upscaling
\cite{dur91, weh02, cdgw03,hn00} for single-phase flow
and pseudo-relative permeability approach
\cite{cdgw03,KB,BT}. The pseudo-relative permeability approach
computes nonlinear relative permeability functions.
These nonlinear approaches are known to lack
robustness and are process dependent \cite{ed01,ed03}. To overcome
these difficulties, one needs a better understanding of nonlinear upscaling
methods for space-time heterogeneous problems.
Nonlinear upscaling methods for scale separation cases
are rigorously treated in
\cite{pankov97, ep03d}.
Among these approaches, some deal with problems that have both space
and time heterogeneities.

Our proposed approaches take their origin in the Constraint Energy Minimizing
Generalized Multiscale
Finite Element Method (GMsFEM) and Nonlocal Multi-Continua upscaling,
which are related. The main idea of these approaches is to use
multiple macroscopic parameters to represent the solution over each
coarse-grid block. We refer to these degrees of freedom as continua,
which are important for achieving a high order accuracy.
We note that generalized continua concepts are also introduced in
computational mechanics \cite{fafalis2012capability},
which include
generalized continuum theories (e.g., \cite{fafalis2012capability}),
computational continua
framework (e.g., \cite{fish2010computational}),
and other approaches.
Computational continua (\cite{fish2010computational,fafalis2018computational}), which use nonlocal quadrature to couple the coarse scale system stated on unions of some disjoint computational unit cells, are introduced for non-scale-separation heterogeneous media. In \cite{fish2012reduced,fish2015computational,fish2013practical}, the
computational continua with model reduction technique is combined.

An important step that connects multiscale methods and upscaling techniques
includes
using basis functions such that the resulting degrees of freedom
have physical meanings, typically averages of the solution.
For nonlinear problems, using linear basis functions is not very suitable.
The local problems are nonlinear problems. For this reason, in our first
work \cite{chung2018nonlinear}, we provided a framework for NLMC for stationary problems.
In this paper, we provide a unified framework for nonlinear NLMC for problems
with space-time heterogeneities, analysis, and machine learning based simplified local solves.

\begin{figure}
\centering
\includegraphics[width=6in]{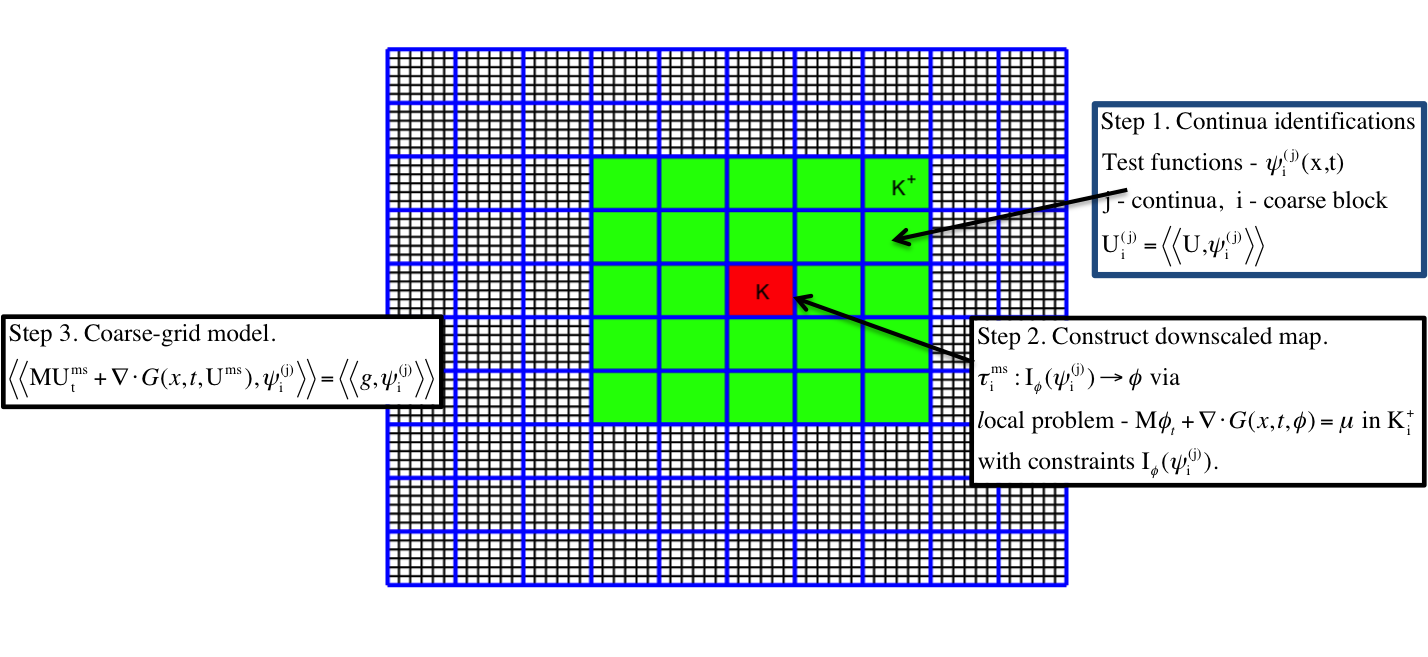}
\caption{Schematic description of the method.}
\label{fig:ill1}
\end{figure}

In Figure \ref{fig:ill1}, we illustrate the main steps of our approach.
Below, we briefly describe them. In the first step, we identify
continua in each coarse block. This is done with the help of
test functions, which can separate the features that can not be
localized within the region of influence (oversampling region designated
with green color in  Figure \ref{fig:ill1}). For nonlinear problems, each continua is defined by a
corresponding test function.
Continua play the role of macroscale variables. In
our examples, macroscale variables are average solution values in some selected heterogeneous
regions (such as channels).

In Step 2, once we identify the continua,
we use oversampling regions to define downscaling maps.
The oversampling region represents
the region of influence and thus, the macroscopic parameter interactions
are defined within oversampling regions.
 The local nonlinear
problems are formulated
in the oversampled regions using constraints. However, these
computations are expensive and require appropriate local
problems. Instead, we propose to use local space-time models
of the original PDEs and perform many tests with various boundary
conditions and sources. These local solutions are used to train macroscopic
parameters as a function of multiple macroscale continua variables. For
machine learning, we use deep learning algorithms, which allow approximating
complex multi-continua dependent functions.

In Step 3, we seek a coarse-grid solution (the values in each continua)
such that the downscaled global fine-scale solution satisfies
the variational formulation that uses the test functions defined in Step 1.
An example of test functions that we use is piecewise constant functions
in each subregions (defined as channels).
 Then, the macroscale variables are average solutions
defined in these subregions. The corresponding downscaled maps represent
the local fine-grid solutions given these constraints. The global
coarse-grid formulation can be thought as a mass balance equation formulated
for each continua.

The main contributions of this paper are the following:
\begin{itemize}

\item  Novel upscaled model for space-time;

\item Unified framework using test functions;

\item Easy local problems and machine learning calculations;

\item Numerical results that uses machine learning and nonlinear
upscaled models.

\end{itemize}

In the paper, we present an analysis of our approach for a model
problem, which consists of heterogeneous p-Laplacian ($p=2$). This model
problem requires nonlinear upscaling and some oversampling in order
to show an optimal convergence of our proposed approach.

In conclusion, the paper is organized as follows. In Section
\ref{sec:nlmc}, we give some preliminary results of the
nonlocal multicontinua approach. In Section \ref{sec:approach}, we present
our approach, which uses the space-time nonlocal multicontinua approach.
In this section, we present examples and convergence results.
The numerical results are presented in Section \ref{sec:numresults}.

\section{Overview of NLMC methods}\label{sec:overview}
\label{sec:nlmc}


In this section, we will give a brief overview of the NLMC method
for linear problems \cite{NLMC}.
Our goal is to summarize the key ideas  and motivate
our new space-time nonlinear NLMC method.
We consider
a model elliptic equation with a heterogeneous coefficient
\begin{equation}\label{eq:parabolic}
- \nabla\cdot (\kappa \nabla u) = f, \quad \text{ in } \Omega.
\end{equation}
Here $\kappa$ is the heterogeneous field,
$f$ is a given source and
$\Omega$ is the physical domain.

The NLMC method is defined on a coarse mesh,
$\mathcal{T}^H$, of the domain $\Omega$.
We write $\mathcal{T}^H = \bigcup \{ K_i \; | \; i=1,\cdots, N\}$, where $K_i$ denotes the $i$-th coarse element
and $N$ denotes the number of coarse elements in $\mathcal{T}^H$.
For each coarse element $K_i$,
we define an oversampled region $K_i^+$,
which is obtained by enlarging
the coarse block $K_i$ by a few coarse grid layers. We will also denote $K_i^+ = K_{i,l}$
when the oversampling region is obtained by enlarging $K_i$ by $l$ coarse grid layers.
See Figure~\ref{grid} for an illustration of coarse grid and
oversample region.
In particular, a structured coarse grid is shown with boundaries
of coarse elements are denoted red.
A coarse cell $K_i$ is denoted green and its oversampled region $K_i^+$ obtained by
enlarging $K_i$ by one coarse grid layer
is enclosed by black lines.

\begin{figure}
\centering
\includegraphics[width=6cm]{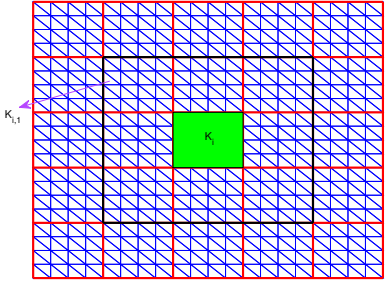}
\caption{Schematic of the coarse grid $K_i$, the oversampling region $K_{i,1}$ and the fine grids.}
\label{grid}
\end{figure}

The NLMC method consists of three main ingredients:
\begin{enumerate}
\item Choice of continua.
\item Local basis functions.
\item Global coupling.
\end{enumerate}

For each coarse element $K_i$, we will identify multiple continua corresponding to various solution features. This can be done
via a local spectral problem or a suitable weight function.
Using the definition of continua, we will define a set of local basis functions by solving some local problems on oversample regions.
Then, the final NLMC system is defined using these multiscale basis functions and a suitable variational formulation.
In the following, we will discuss these concepts in detail.

Now we will specify the definition of continuum that is used in our studies.
 For each coarse block $K_i$,
we will identify a set of continua which are represented by a set of
auxiliary basis functions $\phi^{j}_i$,
where $j$ denotes the $j$-th continuum.
There are multiple ways to construct these functions $\phi^{j}_i$.

One way is to
follow the idea proposed in CEM-GMsFEM \cite{chung2018constraint}.
In this framework, the auxiliary basis functions $\phi^{j}_i$
are obtained as the dominant eigenfunctions of a local spectral problem defined on $K_i$.
These eigenfunctions can capture the heterogeneities and the
contrast of the medium.
We can also follow the framework in the original
NLMC method \cite{NLMC}, designed for flows in fractured media,
which can be easily modified for general heterogeneous media.
In this approach, one identifies explicit information of fracture networks.
The auxiliary basis functions $\phi^{j}_i$ are piecewise constant
functions, namely, they equal one within one fracture network
and zero otherwise.
Moreover, one can define the continua
by using properties of the heterogeneous media.
In this case, the auxiliary basis functions are piecewise constant functions defined with respect to a partition of
the coarse cell $K_i$, such as the medium coefficients have a bounded contrast in each subregion \cite{zhao2019analysis}.



Once the auxiliary basis functions $\phi^{j}_i$ are specified,
we can construct the required basis functions.
The idea generalizes the original energy minimization framework in CEM-GMsFEM.
First, we denote the space of auxiliary basis functions as $V_{aux}$.
Consider a given coarse element $K_i$ and a given continuum $j$ within $K_i$.
We will use the corresponding auxiliary basis function $\phi^{j}_i$
to construct our required multiscale basis function $\psi^{j}_i$
by solving a problem in an oversampled region $K_i^+$.
Specifically, we find $\psi^{j}_i \in H^1_0(K_i^+)$ and $\mu \in V_{aux}$
such that
\begin{equation}\label{eq:basis}
\begin{aligned}
& \int_{K_i^+} \kappa \nabla \psi^{j}_i \cdot \nabla v + \int_{K_i^+} \tilde{\kappa} \mu v = 0, \quad \forall v\in H^1_0(K_i^+), \\
& \int_{K_\ell} \tilde{\kappa} \psi^{j}_{i}\phi_m^{\ell}   = \delta_{j\ell} \delta_{im}, \quad \forall K_\ell \subset K_i^+,
\end{aligned}
\end{equation}
where $\delta_{i m}$ denotes the standard delta function and $\tilde{\kappa}$ is a weight function.
We remark  the function $\mu$ serves as a Lagrange multiplier for the constraints in the second equation of (\ref{eq:basis}).
We also remark that the basis function $\psi^{j}_i$ has mean value one on the $j$-th continuum within $K_i$
and has mean value zero in all other continua in all coarse elements within $K_i^+$.
In practice, the above system (\ref{eq:basis}) is solved in $K_i^+$ using a fine mesh, which is typically a refinement
of the coarse grid. See Figure~\ref{grid} for an illustration.

Finally, we can derive the NLMC system.
Let $V_{ms}$ be the space spanned by the basis functions $\{ \psi^j_i\}$.
We will represent the approximate solution $u_{ms} \in V_{ms}$ as a linear combination of basis functions, namely,
$$
u_{ms} = \sum_{i=1}^N \sum_j U_i^j \psi^j_i.
$$
Then, we will find $u_{ms}$ by the following variational formulation
\begin{equation*}
a(u_{ms}, \psi) = (g,\psi), \quad \forall \psi \in V_{ms}.
\end{equation*}
This variational formulation results in the following upscaled model for the solution $U = ( U^j_i)$:
\begin{equation*}
A_T U = F
\end{equation*}
where
the upscaled stiffness matrix $A_T$ is defined as
\begin{equation}\label{eq:Trans1}
(A_T)_{jm}^{(i,\ell)} = a(\psi_j^{i} ,\psi_m^{\ell} ) := \int_{\Omega} \kappa \nabla \psi^{i}_j \cdot \nabla \psi_m^{\ell},
\end{equation}
and the upscaled source term $F$ is defined as
\begin{equation*}
(F)^{(j)}_i = (g, \psi^j_i).
\end{equation*}
We remark that the nonlocal connections of the continua are coupled by the matrix $A_T$.
We also remark that the local computation in (\ref{eq:basis})
results from a spatial decay property of the multiscale basis function, see \cite{chung2018constraint,chung2018fast,chung2018constraintmixed}
for the theoretical foundation.

The above NLMC idea can be extended to nonlinear elliptic problems, resulting in a
nonlinear NLMC method (\ref{eq:global1})-(\ref{eq:global2}). See Section \ref{sec:non-nlmc}
for the derivation and the convergence analysis.

\section{Nonlinear non-local multicontinua model}
\label{sec:approach}

In this section, we present the nonlinear non-local multicontinua (NLMC) method. We will first
give some general concept of the methodology in Section \ref{sec:concept}. Then, in Section \ref{sec:example}, we
give some illustrative examples including linear problems and pseudomonotone problems.
The main methodological details of the method are presented in Section \ref{sec:detail}.
Finally, we present a convergence analysis of the method for a model elliptic problem in Section \ref{sec:non-nlmc}.

\subsection{General concept}
\label{sec:concept}
We will first present some general concepts of our nonlinear NLMC using the following model nonlinear problem
\begin{equation}
\label{eq:non1}
M U_t + \nabla \cdot G(x, t, U)=g,
\end{equation}
where
$G$ is a nonlinear operator that has a multiscale dependence with respect to
space (and time, in general) and $M$ is a linear operator.
In the above equation, $U$ is the solution and $g$ is a given source term.
Our method has three key ingredients, namely, the choice of continua, the construction of local downscaling map
and the construction of the coarse scale model. We will
summarize these concepts in the following.

\begin{itemize}

\item {\bf The choice of continua}

The continua serve as our macroscopic variables in each coarse element. Our approach uses
a set of test functions to define the continua.
To be more specific, we consider a coarse element $K_i$.
We will choose a set of test functions $\{ \psi_i^{(j)}(x,t) \}$ to define our continua,
where $j$ denotes the $j$-th continuum.  Using these test functions, we can define
our macroscopic variables as
$$
U_i^{(j)} = \langle \langle U, \psi_i^{(j)} \rangle\rangle
$$
where $\langle \langle \cdot,\cdot \rangle\rangle$ is a space-time inner product.


\item {\bf The construction of local downscaling map}

Our upscale model uses a local downscaling map to bring microscopic information to the coarse grid model.
The proposed downscaling map is a function defined on an oversampling region subject to some constraints
related to the macroscopic variables.
In time-dependent problems, the oversampling region can be regarded as a zone
of influence for coarse-grid variables defined on the target coarse block
$K_i$.
More precisely, we consider a coarse element $K_i$,
and an oversampling region $K_i^+$ such that $K_i \subset K_i^+$. Then we
find a function $\phi$ by
solving the following local problem
\begin{equation}
\label{eq:non11}
M \phi_t + \nabla \cdot G(x,t, \phi)=\mu, \quad\text{in } K_i^+.
\end{equation}
The above equation (\ref{eq:non11}) is solved subjected to constraints defined by the following functionals
\[
I_\phi(\psi_i^{(j)}(x,t)).
\]
This constraint fixes some averages of $\phi$ with respect to
$\psi_i^{(j)}(x,t)$.
We remark that the function $\mu$ serves as the Lagrange multiplier for the above constraints. This local solution
 builds a downscaling map
\[
\mathcal{F}_i^{ms}:I_\phi(\psi_i^{(j)}(x,t))\rightarrow \phi.
\]

\item  {\bf The construction of coarse scale model}

We will construct the coarse scale model using the test functions $\{ \psi_i^{(j)}(x,t) \}$
and the local downscaling map. Our upscaling solution $U^{ms}$ is defined
as a combination of the local downscaling maps.
To compute $U^{ms}$, we use the following variational formulation
\begin{equation}
\label{eq:nonlinearNLMC}
\langle  \langle M U^{ms}_t + \nabla \cdot G(x,t, U^{ms}),\psi_i^{(j)} \rangle \rangle= \langle  \langle g,\psi_i^{(j)} \rangle \rangle.
\end{equation}
The above equation (\ref{eq:nonlinearNLMC}) is our coarse scale model.

\end{itemize}

We would like to briefly summarize above steps. The first step defines
multicontinua, which play the role of macroscopic variables. They are
critical in multiscale modeling and need to be defined apriori.
The second step constructs downscaling maps and can be
computationally intensive. We will propose a machine learning
technique in combination with solving local problems of the original
equation subject to various boundary conditions. From here, the macroscale
fluxes will be defined as a function of macroscopic variables
in oversampled regions. This high dimensional functions will be learned
using machine learning techniques during coarse-grid solution step
(Step 3).
Next, we will give some examples (Section \ref{sec:example}) and then present a more detailed description of the algorithm (Section \ref{sec:detail}).

\subsection{Examples}
\label{sec:example}

We will present two model problems, and discuss how our nonlinear NLMC is applied.

\subsubsection{Linear case}

In this section, we will construct our upscaling model for
a case that $G$ is a linear operator.
We will follow the general concepts in Section \ref{sec:concept}.
First, we discuss the choice of continua. For each coarse element $K_i$, we consider a set of
test functions $\{ \psi^{(j)}_i(x,t)\}$ defined for $x\in K_i$. Here the index $j$ denotes the $j$-th continuum.
One choice of these test functions is a set of piecewise constant functions. Another choice of these test functions
is the first $j$ dominant eigenfunctions of an appropriate spectral problem.

Next, we discuss the construction of the local downscaling map.
We fix a continuum $\psi^{(j)}_i(x,t)$ in the coarse region $K_i$.
Let $K_i^+$ be an oversampling region.
With the assumption that $G$ is linear, we can represent the downscaling map, denoted by $\phi^{(j)}_i$,  
as a linear combination of some generic local solutions $\{ \phi^{(j,l)}_{i,m}\}$.
To find these functions $\{ \phi^{(j,l)}_{i,m}\}$, we solve the following
\begin{equation}
\begin{split}
M(\phi^{(j,l)}_{i,m})_t +\nabla \cdot G(x,t,\phi^{(j,l)}_{i,m}) = \mu^{(j,l)}_{i,m} \\
\langle  \langle \phi^{(j,l)}_{i,m},\psi_s^{(r)} \rangle \rangle= \delta_{lr}\delta_{ms}
\end{split}
\end{equation}
on the oversample region $K_i^+$, where $\langle \langle \cdot,\cdot \rangle \rangle$ is an inner product
and $ \mu^{(j,l)}_{i,m}$ plays the role of Lagrange multiplier. Using these functions $\{ \phi^{(j,l)}_{i,m}\}$,
we can represent the local downscaling map $\psi^{(j)}_i(x,t)$ as
$$
\psi^{(j)}_i(x,t) = \sum_{m,l} U^{(l)}_m \phi^{(j,l)}_{i,m}.
$$
Since $G$ is linear, we have
\[
G(x,t,\sum_{m,l} U^{(l)}_m\phi^{(j,l)}_{i,m})= \sum_{m,l} U^{(l)}_m G(x,t,\phi^{(j,l)}_{i,m}).
\]
Let $\{ \chi_i\}$ be a set of partition of unity functions corresponding to the partition $\{ K_i^+\}$ of the domain $\Omega$.
The final upscale solution is then defined as the combination $\phi := \sum_i \sum_j \chi_i \phi_i^{(j)}$.
Using the test functions $\psi_i^{(j)}$, we can compute the macroscopic value $\{ U_m^{(l)}\}$ by
the following variational formulation 
\begin{equation}
\langle \langle M\phi_t +\nabla \cdot G(x,t,\phi),  \psi^{(j)}_i \rangle \rangle = \langle \langle g, \psi_i^{(j)} \rangle \rangle,
\quad \forall \psi_i^{(j)}.
\end{equation}



\subsubsection{Pseudomonotone case}


Next, we consider another example for which $G$ is a pseudo-monotone operator. In this case, to compute the downscaling map, $\mathcal{F}^{ms}$, we will need to solve the following local problem: find $\mathcal{F}^{ms}(U)$ and $\mu$ such that
\begin{equation}
\begin{split}
M(\mathcal{F}^{ms}(U))_t +\nabla \cdot G(x,t,\mathcal{F}^{ms}(U)) = \mu\\
\langle  \langle \mathcal{F}^{ms}(U),\psi_i^{(j)} \rangle \rangle= U^{(j)}_i
\end{split}
\end{equation}

The coarse grid system is then defined as
\begin{equation}
\sum_{l,m} \langle  \langle M(\mathcal{F}^{ms}(U))_t  +\nabla \cdot G(x,t,\mathcal{F}^{ms}(U)),\psi_i^{(j)} \rangle  \rangle = \langle  \langle g,\psi_i^{(j)}(x,t) \rangle \rangle \;\forall \psi_i^{(j)}.
\end{equation}

\subsection{More details of general framework}
\label{sec:detail}

In this section, we give the details of our nonlinear NLMC framework.
We consider the following model problem of finding $u\in V$ such that
\[
\partial_{t}u+L(u)=f, \quad\text{in } \Omega \times (0,T]
\]
 with $u(\cdot,0)=0$, where $L$ is a nonlinear differential operator, $T>0$ is a fixed time and $V$ is a suitable function space.
We use a different notation for nonlinear differential operator as in (\ref{eq:non1}) to simplify the notations,
and our methodology remains applicable to the problem described by (\ref{eq:non1}).

Next, we discuss the mesh.
We assume that $\Omega$ is partitioned by a coarse mesh $\mathcal{T}_{H}$ (see Figure \ref{grid}) with mesh size $H>0$
and $(0,T]$ is partitioned into coarse time intervals denoted as $\mathcal{T}_{T}=\{(t_{i},t_{i+1}]\}$.
A space-time element $K^{(n,i)}$ is then defined by $K_i \times (t_{n},t_{n+1}]  $
for a coarse cell $K_{i}\in\mathcal{T}_{H}$ and the $n$-th time interval $(t_n, t_{n+1}]$.
The construction of our nonlinear NLMC method follows the three steps explained in Section \ref{sec:concept}.

{\bf Approximation by global basis}

The discussion of our method starts with the use of global basis functions. In this case, the basis functions
are global in space and in time. The motivation of this follows from the global basis of CEM-GMsFEM \cite{chung2018constraint},
for which coarse grid convergence is obtained.

\begin{itemize}

\item {\bf Choice of continua}

The continua is defined using a set of test functions. Consider a space-time element $K^{(n,i)}$,
we will introduce a set of test functions $V_{aux} = \{ \psi_{j}^{(n,i)} \}$ which
corresponding to different continua of the problem.
We notice that
$\psi_{j}^{(n,i)}$ is supported in $K^{(n,i)}$.
We let $N_c$ be the number of such test functions.
Then
we will define macroscopic variables
by
\[
U_j^{(n,i)}=s(u,\psi_{j}^{(n,i)})=\int_{0}^{T}\int_{\Omega}\tilde{\kappa}\psi_{j}^{(n,j)}u
\]
where
$s(\cdot,\cdot)$ is a weighted $L^{2}$ inner product with
weighting function $\tilde{\kappa}$ such that $c_{0}H^{-1}\leq\tilde{\kappa}\leq c_{0}H^{-1}$.
Note that this condition for the weighting function is motivated by the weighting function used in CEM-GMsFEM.

\item {\bf Global downscaling map}

We will define a downscaling map. This downscaling map will give a function defined globally in space and in time
with constraints defined using a given set of macroscopic values. More precisely, we fix a set of macroscopic values $\{ U_j^{(n,i)}\}$.
We will then define a function $F = (F_1, F_2)$
such that $F_1\in V$ and $F_2 \in V_{aux}$. These functions are obtained by solving
\begin{align*}
\int_{0}^{T}\int_{\Omega}(\partial_{t}+L)F_{1}(U)v-s(F_{2}(U),v) & =0,\quad\forall v\in V, \\
s(F_{1}(U),\psi_{j}^{(n,i)}) & =U_{j}^{(n,i)}, \quad\forall\psi_{j}^{(n,i)} \in V_{aux}.
\end{align*}
We notice that the global function $F_1$ has macroscopic values equal to the given values $\{ U_j^{(n,i)}\}$
and the function $F_2$ serves as the Lagrange multiplier for these constraints.

\item {\bf Coarse grid model}

Next, using the downscaling map,
we can define the
global coarse grid problem as: finding $U\in\mathbb{R}^{N_c}$ such
that
\[
s(F_{2}(U),\psi_{j}^{(n,i)})=\int_{0}^{T}\int_{\Omega}f\psi_{j}^{(n,i)}, \quad \forall\psi_{j}^{(n,i)} \in V_{aux}.
\]
Then,
the global numerical solution $u_{glo}$ is defined by $u_{glo}=F_{1}(U)$.

\end{itemize}

{\bf Nonlinear NLMC method}

Now we will present the nonlinear NLMC method. The key ingredient is that we will replace
the global downscaling map above by a local downscaling map.

\begin{itemize}

\item {\bf Local downscaling map}

We will introduce the localized downscaling operator
$F_{ms} = (F_{ms,1},F_{ms,2})$. Consider a space-time element $K^{(n,i)}$.
We define a space-time oversampling region $K_+^{(n,i)} = K_i^+ \times (t_{n}^{-},t_{n+1}]$
where $t_n^- < t_n$.
We will then define a function $F_{loc}^{(n,i)} = (F_{loc,1}^{(n,i)}, F_{loc,2}^{(n,i)})$
such that $F_{loc,1}^{(n,i)} \in V(K_+^{(n,i)})$ and $F_{loc,2}^{(n,i)} \in V_{aux}(K_+^{(n,i)})$,
where $ V(K_+^{(n,i)})$ and $V_{aux}(K_+^{(n,i)})$ are restrictions of $V$ and $V_{aux}$ on $K_+^{(n,i)}$ respectively.
These functions are obtained by solving
\begin{align*}
\int_{t_{n}^{-}}^{t_{n+1}}\int_{K_{+}^{i}}(\partial_{t}+L)F_{loc,1}^{(n,i)}(U)v-s(F_{loc,2}^{(n,i)}(U),v) & =0, \quad\forall v\in V(K_{+}^{(n,i)}), \\
s(F_{loc,1}^{(n,i)}(U),\psi_{j}^{(n,i)}) & =U_{j}^{(n,i)}, \quad\forall\psi_{j}^{(n,i)}\in V_{aux}(K_{+}^{(n,i)}).
\end{align*}
Finally the localized downscale
operator is defined by $F_{ms,p}(U)=\sum_{n,l}\chi^{(n,i)}F_{loc,p}^{(n,i)}(U)$
where $p=1,2$ and
$\chi^{(n,i)}$ is a partition of unity such that $\sum_{n,i}\chi^{(n,i)}\equiv1.$

\item {\bf Coarse grid model}

The coarse grid problem is then defined as: finding $U\in\mathbb{R}^{N_c}$
such that

\[
s(F_{ms,2}(U),\psi_{j}^{(n,i)})=\int_{0}^{T}\int_{\Omega}f\psi_{j}^{(n,i)}, \quad\forall\psi_{j}^{(n,i)}\in V_{aux}
\]
and the nonlinear NLMC solution $u_{ms}$ is defined by $u_{ms}=F_{ms,1}(U)$.

\end{itemize}

\subsection{Error sources and analysis}
\label{sec:non-nlmc}
In this section, we present a concept of the analysis for the method.
We will use a simple monotone elliptic equation to illustrate the main ideas.
We consider the following problem: find $u$ such that
\begin{equation}
\label{eq:monotonePDE}
\begin{aligned}
\nabla\cdot(\kappa(x,\nabla u)) & =f,  \quad &&\text{in } \Omega, \\
u & =0,  \quad &&\text{on } \partial\Omega,
\end{aligned}
\end{equation}
where $\kappa(x,v)$ is a heterogeneous function.
The weak formulation
of the above equation can be written as: find $u\in V=H_{0}^{1}(\Omega)$
such that
\begin{align*}
A_{\Omega}(u,w) & =\int_{\Omega}fw, \quad \forall w\in H_{0}^{1}(\Omega),
\end{align*}
 where, for any open subset $\omega\subset \Omega$ of the domain,
  the operator $A_{\omega}$ is defined by
\[
A_{\omega}(u,w)=\int_{\omega}\kappa(x,\nabla u)\cdot\nabla w.
\]
We will assume that the heterogeneous function $\kappa(x,v)$ satisfies the following two properties.

\noindent
{\bf Assumption on $\kappa(x,v)$}

\begin{enumerate}
\item If the vector field $v = 0$, then $\kappa(x,v)=0$.

\item Lipschitz continuity with respect to $v$:

We assume there exist a function $\overline{\kappa}\in L^{\infty}(\Omega)$ such that
\begin{equation}
|\kappa(x,z)-\kappa(x,v)|\leq C_{1} \, \overline{\kappa}(x)|z-v|.
\label{eq:upper_bound}
\end{equation}

\item Monotonicity:

We assume that the following coercivity condition holds
\begin{equation}
\kappa(x,v)\cdot v\geq C_{2}\,\overline{\kappa}(x)|v|^{2}.
\label{eq:lower_bound}
\end{equation}

\end{enumerate}

Next,  for any open subset $\omega\subset \Omega$ of the domain,
we define two inner products $a_{\omega}(\cdot,\cdot)$ and $s_{\omega}(\cdot,\cdot)$
as follows
\[
a_{\omega}(u,w)=\int_{\omega}\overline{\kappa}\nabla u\cdot\nabla w \quad \text{and }\quad s_{\omega}(u,w)=\int_{\omega}\tilde{\kappa}uw
\]
where $\tilde{\kappa}(x) =\overline{\kappa}\sum_{i}|\nabla\chi_{i}|^{2}$ and $\{\chi_{i}\}_{i=1}^{N}$ is a set of partition of unity functions corresponding to the coarse mesh such that
$0\leq\chi_{i}\leq1$.
The norms $\|\cdot\|_{a(\omega)}$ and $\|\cdot\|_{s(\omega)}$ corresponding
to these inner products are defined as
\[
\|u\|_{a(\omega)}^{2}=a_{\omega}(u,u) \quad\text{and}\quad \|u\|_{s(\omega)}^{2}=s_{\omega}(u,u)
\]
respectively.
To simplify the notation, we use $A$, $\|\cdot\|_{a}$ and $\|\cdot\|_{s}$ to denote
$A_{\Omega}$, $\|\cdot\|_{a(\Omega)}$ and  $\|\cdot\|_{s(\Omega)}$ respectively.

In the following Lemma, we will show that the operator $A_{\omega}$
satisfies some coercivity and continuity properties.

\begin{lemma}
\label{lem:coercive}
For $\omega\subset \Omega$, $u,v,w\in H^{1}(\omega)$, we have
\[
A_{\omega}(u,u)\geq C_{2}\|u\|_{a(\omega)}^{2}
\]
and
\[
\Big|A_{\omega}(u,w)-A_{\omega}(v,w)\Big|\leq C_{1}\|u-v\|_{a(\omega)}\|w\|_{a(\omega)}.
\]
Moreover, we have
\[
\Big|A_{\omega}(u,w)\Big|\leq C_{1}\|u\|_{a(\omega)}\|w\|_{a(\omega)}.
\]
\end{lemma}

\begin{proof}
By the assumption (\ref{eq:lower_bound}), we have
\[
A_{\omega}(u,u)=\int_{\omega}\kappa(x,\nabla u)\cdot\nabla u\geq C_{2}\int_{\omega}\overline{\kappa}(x)|\nabla u|^{2}=C_{2}\|u\|_{a(\omega)}^{2}.
\]
which proves the first inequality.
By the assumption (\ref{eq:upper_bound}), we have
\begin{align*}
\Big|A_{\omega}(u,w)-A_{\omega}(v,w)\Big| & =\Big|\int_{\omega}\Big(\kappa(x,\nabla u)-\kappa(x,\nabla v)\Big)\cdot\nabla w\Big|\leq\int_{\omega}\big|\kappa(x,\nabla u)-\kappa(x,\nabla v)\big|\cdot\big|\nabla w\big|\\
 & \leq C_{1}\int_{\omega}\overline{\kappa}\big|\nabla(u-v)\big|\cdot\big|\nabla w\big|\leq C_{1}\|u-v\|_{a(\omega)}\|w\|_{a(\omega)}
\end{align*}
which gives the second inequality.
Recall that $\kappa(x,0)=0$. Thus we have
\[
\Big|A_{\omega}(u,w)\Big|=\Big|A_{\omega}(u,w)-A_{\omega}(0,w)\Big|\leq C_{1}\|u\|_{a(\omega)}\|w\|_{a(\omega)}
\]
which shows the third inequality. This completes the proof of this lemma.
\end{proof}

We next prove the following technical result.

\begin{lemma}
For $\omega\subset \Omega$, $u,v\in H^{1}(\omega)$, we have
\[
\Big|A_{\omega}(u,v)\Big|=\Big|\int_{\omega}\kappa(x,\nabla u)\cdot\nabla(v\chi_{i})\Big|\leq C_{1}\|u\|_{a(\omega)} \Big(\|v\|_{a(\omega)}+\|v\|_{s(\omega)} \Big).
\]
\end{lemma}

\begin{proof}
Notice that $\nabla(\chi_{i}v)=v\nabla\chi_{i}+\chi_{i}\nabla v$. Thus, we have
\[
\int_{\omega}\kappa(x,\nabla u)\cdot\nabla(v\chi_{i})=\int_{\omega}v\kappa(x,\nabla u)\cdot\nabla\chi_{i}+\int_{\omega}\chi_{i}\kappa(x,\nabla u) \cdot \nabla v.
\]
We will first estimate the term $\int_{\omega}v\kappa(x,\nabla u)\cdot\nabla\chi_{i}$.
By the Cauchy-Schwarz inequality, we have
\[
|\int_{\omega}v\kappa(x,\nabla u)\cdot\nabla\chi_{i}|\leq\Big(\int_{\omega}\overline{\kappa}^{-1}|\kappa(x,\nabla u)|^{2}\Big)^{\frac{1}{2}}\Big(\int_{\omega}\overline{\kappa}|\nabla\chi_i|^{2}|v|^{2}\Big)^{\frac{1}{2}}
\]
and, by using (\ref{eq:upper_bound}), we have
\[
\overline{\kappa}^{-1}|\kappa(x,\nabla u)|^{2}\leq C_{1}^{2}\overline{\kappa}|\nabla u|^{2}.
\]
Combining the above, we have
\[
|\int_{\omega}v \kappa(x,\nabla u)\cdot\nabla\chi_i  |\leq C_{1}\|u\|_{a(\omega)}\|v\|_{s(\omega)}.
\]
To estimate the second term $\int_{\omega}\chi_{i}\kappa(x,\nabla u)\nabla v$,
we use the fact that $|\chi_i| \leq 1$ and assumption  (\ref{eq:upper_bound}) to obtain
\begin{align*}
\Big|\int_{\omega}\chi_{i}\kappa(x,\nabla u)\nabla v\Big| & \leq\int_{\omega}|\kappa(x,\nabla u)||\nabla v|\leq C_{1}\int_{\omega}\overline{\kappa}|\nabla u||\nabla v|
  \leq C_{1}\|u\|_{a(\omega)}\|v\|_{a(\omega)}.
\end{align*}
This completes the proof of this lemma.
\end{proof}

In the following, we will formulate our nonlinear NLMC method for the equation (\ref{eq:monotonePDE}).
The coarse scale degrees of freedom (continua) $\overline{U}$
of the solution is defined as
\[
\overline{U}_{i,j}=\int_{\Omega}\tilde{\kappa}u\mu_{i,j}
\]
 for some $\mu_{i,j}\in L^{\infty}(\Omega)$ where $\mu_{i,j}|_{K_{m}}=0$ if $m\neq i$. The auxiliary space $V_{aux}$
is then defined as
\[
V_{aux}=\text{span}_{i,j}\{\mu_{i,j}\}
\]
We remark that the functions in $V_{aux}$ defines the continua. In particular, $\mu_{i,j}$
defines the $j$-th continuum in the coarse cell $K_i$.

Next, to construct the numerical upscaling equation for (\ref{eq:monotonePDE}), we will
define a global downscaling operator $F_1$ such that $F_1(\overline{U})\in H_{0}^{1}(\Omega)$ and
\begin{eqnarray}
\int_{\Omega}\kappa(x,\nabla F_1(\overline{U})) \cdot \nabla v-s(F_{2}(\overline{U}),v) & =0, \quad \forall v\in H_{0}^{1}(\Omega), \label{eq:global1} \\
s(F_{1}(\overline{U}),\mu_{i,j}) & =\bar{U}_{i,j}, \quad\forall \mu_{i,j}\in V_{aux}. \label{eq:global2}
\end{eqnarray}
Next, we will define a projection operator $\Pi:V\rightarrow V_{aux}$
such that
\[
\int_{\Omega}\tilde{\kappa}\Pi(u)\mu=\int_{\Omega}\tilde{\kappa}u\mu, \quad \forall\mu\in V_{aux}.
\]
The global solution $U_{glo}\in V_{aux}$ is defined by
\[
\int_{\Omega}F_{2}(U_{glo})v=\int_{\Omega}fv, \quad\forall v\in V_{aux}
\]
and the global downscaled solution $u_{glo}$ is defined as $u_{glo}=F_{1}(U_{glo})$.

\noindent
{\bf Approximation by global basis}

We summarize the main steps:
\begin{enumerate}
\item Find $U_{glo}\in V_{aux}$
\begin{equation}
\label{eq:nlmc1}
\int_{\Omega}F_{2}(U_{glo})v=\int_{\Omega}fv, \quad\forall v\in V_{aux}.
\end{equation}

\item Define
\begin{equation}
\label{eq:nlmc2}
u_{glo}=F_{1}(U_{glo}).
\end{equation}

\end{enumerate}

Next, we will construct the nonlinear NLMC method.
For each $K\in\mathcal{T}_{H}$, we will define a local downscaling
operator $F^{loc,K}$ such that $F_{1}^{loc,K}(\bar{U})\in H_{0}^{1}(K^{+})$ and

\begin{align*}
\int_{K^{+}}\kappa(x,\nabla F_{1}^{loc,K}(\bar{U}))\cdot\nabla v-\int_{K^{+}}\tilde{\kappa}F_{2}^{loc,K}(\bar{U})v & =0, \quad \forall v\in H^{1}(K^{+}), \\
\int_{K^{+}}\tilde{\kappa}F_{1}^{loc,K}(\bar{U})\mu_{i,j} & =\bar{U}_{i,j}, \quad\text{for }\mu_{i,j}\in V_{aux}.
\end{align*}
The multiscale solution $U_{ms}\in V_{aux}$ is defined by
\[
\sum_{K}\int_{K}F_{2}^{loc,K}(U_{ms})v=\int_{\Omega}fv, \quad \forall v\in V_{aux}
\]
and the downscaled multiscale solution $u_{ms}$ is defined as $u_{ms}=F_{1}^{ms}(U_{ms}):=\sum_{K}\chi_{K}F_{1}^{loc,K}(U_{ms})$
where $\chi_{K}$ is a partition of unity such that $\sum_{K}\chi_{K}\equiv1$
with $\sum_{K}|\nabla\chi_{K}|^{2}\leq C\sum_{i}|\nabla\chi_{i}|^{2}$
and $\text{supp}\{\chi_{K}\}\subset K^{+}=\cup_{\overline{K}_{i}\cap\overline{K}\neq\emptyset}\overline{K}_{i}.$

\noindent
{\bf Nonlinear NLMC method}

We summarize the main steps:
\begin{enumerate}
\item Find $U_{ms}\in V_{aux}$
\begin{equation}
\label{eq:nlmc3}
\sum_{K}\int_{K}F_{2}^{loc,K}(U_{ms})v=\int_{\Omega}fv, \quad \forall v\in V_{aux}.
\end{equation}

\item Define
\begin{equation}
\label{eq:nlmc4}
u_{ms}=F_{1}^{ms}(U_{ms}).
\end{equation}

\end{enumerate}

The analysis of our scheme is based on three assumptions. We summarize them below.

\noindent
\textbf{Assumption 1:} For all $K\in\mathcal{T}_{H}$, $v\in V(K)$,
we have
\[
\cfrac{\|(I-\pi)v\|_{s}}{\|v\|_{a}}\leq C H.
\]
\textbf{Assumption 2:} For $K\in\mathcal{T}_{H}$, $v_{aux}\in V_{aux}(K)$,
there exist a function $w\in H_{0}^{1}(K)$ such that
\[
\|v_{aux}\|_{s(K)}^{2}\leq s_K(v_{aux},w),\quad \|w\|_{a(K)}\leq C\|v_{aux}\|_{s(K)}.
\]
\textbf{Assumption 3:} There exist a $C_{\kappa}>0$ such that
\[
\|v\|_{s}\leq C_{\kappa}\|v\|_{a}, \quad \forall v\in V.
\]

We will prove the following lemma
for the stability of the downscale map.

\begin{lemma}
\label{lem:stablity}By assumption 2, we have
\[
\|F_{1}(\overline{U})\|_{a}\leq CC_{1}C_{2}^{-1}\|\overline{U}\|_{s}
\]
and
\[
\|F_{1}^{loc,K}(\overline{U})\|_{a}\leq C C_1 C_2^{-1} \|\overline{U}\|_{s}
\]
\end{lemma}
\begin{proof}
First, by Lemma \ref{lem:coercive} and (\ref{eq:global1}), we have
\[
\|F_{1}(\overline{U})\|_{a}^{2}\leq C_{2}^{-1}A(F_{1}(\overline{U}),F_{1}(\overline{U}))=s(F_{2}(\overline{U}),F_{1}(\overline{U}))
\]
and by (\ref{eq:global2}), we have
\[
s(F_{2}(\overline{U}),F_{1}(\overline{U}))=s(\overline{U},F_{2}(\overline{U})).
\]
Therefore, we have
\[
\|F_{1}(\overline{U})\|_{a}^{2}\leq C_{2}^{-1}\|\overline{U}\|_{s}\|F_{2}(\overline{U})\|_{s}.
\]
By Assumption 2, there exist a function $w\in H_{0}^{1}(K)$ such that
\begin{equation*}
\|F_{2}(\overline{U})\|_{s(K)}^{2}  \leq s(F_{2}(\overline{U}),w)
\quad\text{and}\quad
\|w\|_{a} \leq \|F_{2}(\overline{U})\|_{s(K)}.
\end{equation*}
Hence, we have
\begin{align*}
\|F_{2}(\overline{U})\|_{s(K)}^{2} & \leq s(F_{2}(\overline{U}),w)=\int_{\Omega}\kappa(x,\nabla F_{1}(\overline{U}))\cdot\nabla w\\
 & \leq C_{1}\|F_{1}(\overline{U})\|_{a}\|w\|_{a}\leq CC_{1}\|F_{2}(\overline{U})\|_{s(K)}\|F_{1}(\overline{U})\|_{a}.
\end{align*}
This shows the first required inequality.
Using a similar argument, we can prove that
\[
\|F_{1}^{loc,K}(\overline{U})\|_{a}\leq CC_{1}C_{2}^{-1}\|\overline{U}\|_{s}.
\]
This completes the proof of the lemma.
\end{proof}

In the following lemma, we will give an error bound for the solution $F_1(U_{glo})$.

\begin{lemma}
\label{lem:gloerror}
Let $u$ be the solution of (\ref{eq:monotonePDE}) and $F_1(U_{glo})$ be the solution of (\ref{eq:nlmc1})-(\ref{eq:nlmc2}).
We have
\begin{align*}
\|u-F_{1}(U_{glo})\|_{a} & \leq CC_2^{-1} H\|(I-\Pi) (f \tilde{\kappa}^{-1}) \|_{s}.
\end{align*}
\end{lemma}

\begin{proof}
First of all, we note that $F_{2}(U_{glo})=\Pi(f \tilde{\kappa}^{-1})$.
So, we have
\[
A(F_{1}(U_{glo}),v)=\int_{\Omega}F_{2}(U_{glo})v=s(\Pi( f \tilde{\kappa}^{-1}),v), \quad \forall v \in H^1_0(\Omega)
\]
 and
\begin{align*}
A(u,v)-A(F_{1}(U_{glo}),v) & =s(f\tilde{\kappa}^{-1}-\Pi( f\tilde{\kappa}^{-1}),v),   \quad \forall v \in H^1_0(\Omega).
\end{align*}
Therefore, by (\ref{eq:lower_bound}), we have
\begin{align*}
C_{2}\|u-F(\Pi u)\|_{a}^{2} & \leq A(u,u-F(U_{glo}))-A(F(U_{glo}),u-F(U_{glo}))\\
 & =s\Big(f\tilde{\kappa}^{-1}-\Pi (f\tilde{\kappa}^{-1}) ,(I-\Pi)(u-F(U_{glo})\Big)\\
 & \leq\| f\tilde{\kappa}^{-1} -\Pi (f\tilde{\kappa}^{-1})\|_{s}\|(I-\Pi)(u-F(U_{glo})\|_{s} \\
& \leq CH\| f\tilde{\kappa}^{-1} -\Pi (f\tilde{\kappa}^{-1}) \|_{s}\|(u-F(U_{glo})\|_{a}
\end{align*}
where the last inequality follows from Assumption 1. This completes the proof.
\end{proof}

In the next lemma, we give a localization result.
To do so, we need some notations for the oversampling domain and the cutoff function with respect to these oversampling domains.
For each coarse cell $K$, we denote $K_{m}^+ \subset \Omega$ as the oversampling coarse region by enlarging
$K$ by $m$ coarse grid layers. For $M>m$, we define $\chi_{M,m}\in\text{span}\{\chi_{i}\}$
such that $0 \leq \chi_{M,m} \leq 1$ and
\begin{align}
\chi_{M,m} & =1,\text{ in }K_{m}^+, \label{cutoff1} \\
\chi_{M,m} & =0,\text{ in }\Omega\backslash K_{M}^+. \label{cutoff2}
\end{align}
Note that, we have $K_{m}^+ \subset K_{M}^+$.

\begin{lemma}
\label{lem:local_estimate}
Assume $K_M^{+}$ is an oversampling region obtained by enlarging the coarse cell $K$ by $M$ coarse grid layers.
Let $\eta_{i}=F_{i}(\overline{U})-F_{i}^{loc,K}(\overline{U})$.
We have
\begin{align*}
\|\eta_{1}\|_{a(K)}^{2} & \leq(1-C^{-1}C_{1}^{-1}C_{2})^{M}\|\eta_{1}\|_{a(K_{M}^{+})}^{2}
\end{align*}
and
\begin{align*}
\|\eta_{2}\|_{s(K)}^{2} & \leq CC_{1}^{2}\|\eta_{1}\|_{a(K)}^{2}\\
 & \leq CC_{1}^{2}(1-C^{-1}C_{1}^{-1}C_{2})^{M}\|\eta_{1}\|_{a(K_{M}^{+})}^{2}.
\end{align*}
\end{lemma}

\begin{proof}
The first step of the proof is to show the following inequality
\begin{equation}
\label{eq:recursive}
\int_{K_{m+1}^{+}}\overline{\kappa}|\nabla(F(\bar{U})-F_{K}^{loc}(\bar{U}))|^{2}\leq C\int_{K_{m+1}^{+}\backslash K_{m}^{+}}\overline{\kappa}|\nabla(F(\bar{U})-F_{K}^{loc}(\bar{U}))|^{2},
\quad\forall m\leq M.
\end{equation}
To do so,
we denote $f_{i}=F_{i}(\overline{U})$ and $g_{i}=F_{i}^{loc,K}(\overline{U})$. By (\ref{eq:lower_bound}), we
obtain
\[
C_{2}\int_{K_{m+1}^{+}}\kappa|\nabla(f_{1}-g_{1})|^{2}\leq\int_{K_{m+1}^{+}}\Big(\kappa(x,\nabla f_{1})-\kappa(x,\nabla g_{1})\Big)\cdot\nabla\Big(f_{1}-g_{1}\Big).
\]
Recalling that $\eta_{i}=f_{i}-g_{i}$. We notice that
\[
\int_{K_{m+1}^{+}}\Big(\kappa(x,\nabla f_{1})-\kappa(x,\nabla g_{1})\Big)\cdot\nabla\Big(\chi_{m+1,m}\eta_{1}\Big)=s(\eta_{2},\chi_{m+1,m}\eta_{1}).
\]
Therefore, we have
\begin{align*}
C_{2}\int_{K_{m+1}^{+}}\overline{\kappa}|\nabla(f_{1}-g_{1})|^{2} & \leq s(\eta_{2},\chi_{m+1,m}\eta_{1})+\int_{K_{m+1}^{+}}\Big(\kappa(x,\nabla f_{1})-\kappa(x,\nabla g_{1})\Big)\cdot\nabla\Big((1-\chi_{m+1,m})\eta_{1}\Big)\\
 & =\int_{K_{m+1}\backslash K_{m}}\tilde{\kappa}\eta_{1}\chi_{m+1,m}\eta_{2}-\int_{K_{m+1}\backslash K_{m}}\Big(\kappa(x,\nabla f_{1})-\kappa(x,\nabla g_{1})\Big)\cdot\nabla\Big(\chi_{m+1,m}\eta_{1}\Big).
\end{align*}
Next, we define $K_{m}^{'}=K_{m+1}^{+}\backslash K_{m}^{+}$ and obtain
\begin{align*}
 & -\int_{K_{m}^{'}}\Big(\kappa(x,\nabla f_{1})-\kappa(x,\nabla g_{1})\Big)\cdot\nabla\Big(\chi_{m+1,m}\eta_{1}\Big)\\
\leq & C_{1}\int_{K_{m}^{'}}\overline{\kappa}|\nabla f_{1}-\nabla g_{1}|\cdot\Big(\eta_{1}\nabla\chi_{m+1,m}+\chi_{m+1,m}\nabla\eta_{1}\Big)\\
\leq & C_{1}\|\eta_{1}\|_{a(K_{m}^{'})} \Big(\|\eta_{1}\|_{a(K_{m}^{'})}+\|\eta_{1}\|_{s(K_{m}^{'})} \Big).
\end{align*}
Therefore, we have
\begin{equation}
\label{eq:est1}
\begin{split}
C_{2}\int_{K_{m+1}^{+}}\overline{\kappa} |\nabla(f_{1}-g_{1})|^{2} & \leq C_{1}\|\eta_{1}\|_{a(K_{m}^{'})}(\|\eta_{1}\|_{a(K_{m}^{'})}+\|\eta_{1}\|_{s(K_{m}^{'})})+\int_{K_{m+1}\backslash K_{m}}\tilde{\kappa}\eta_{1}\chi_{m+1,m}\eta_{2}\\
 & \leq CC_{1}\|\eta_{1}\|_{a(K_{m}^{'})}^{2}+(\int_{K_{m+1}\backslash K_{m}}\tilde{\kappa}\eta_{1}^{2})^{\frac{1}{2}}(\int_{K_{m+1}\backslash K_{m}}\tilde{\kappa}\eta_{2}^{2})^{\frac{1}{2}}\\
 & \leq CC_{1}\|\eta_{1}\|_{a(K_{m}^{'})}^{2}+C\|\eta_{1}\|_{a(K_m')}\|\eta_{2}\|_{s(K_m')}.
\end{split}
\end{equation}
Next we will estimate $\|\eta_{2}\|_{s}$. By Assumption 2, for $K\in\mathcal{T}_{H}$,
there exist a $v\in H_{0}^{1}(K)$ such that
\[
\|\eta_{2}\|_{s(K)}\leq s(\eta_{2},v) \quad\text{and}\quad \|v\|_{a}\leq C\|\eta_{2}\|_{s(K)}.
\]
Thus, for $K\subset K_M^{+}$, we have
\begin{equation}
\label{eq:est2}
\begin{split}
\|\eta_{2}\|_{s(K)} & \leq\cfrac{C\int_{K}\Big( \kappa(x,\nabla f_{1})-\kappa(x,\nabla g_{1})\Big)\cdot\nabla v}{\|v\|_{a}}
  \leq CC_{1}\|\eta_{1}\|_{a(K)}.
\end{split}
\end{equation}
Combining (\ref{eq:est1}) and (\ref{eq:est2}), we have
\[
C_{2}\|\eta_{1}\|_{a(K^+_{m+1})}^{2}\leq CC_{1}\|\eta_{1}\|_{a(K_{m}^{'})}^{2}.
\]
This shows (\ref{eq:recursive}).

By using (\ref{eq:recursive}), we have
\begin{align*}
\|\eta_{1}\|_{a(K_{m}^{+})}^{2} & =\|\eta_{1}\|_{a(K_{m+1}^{+})}^{2}-\|\eta_{1}\|_{a(K_{m}^{'})}^{2}\\
 & \leq(1-C^{-1}C_{1}^{-1}C_{2})\|\eta_{1}\|_{a(K_{m+1}^{+})}^{2}
\end{align*}
and we therefore obtain
\[
\|\eta_{1}\|_{a(K)}^{2}\leq(1-C^{-1}C_{1}^{-1}C_{2})^{m}\|\eta_{1}\|_{a(K_{m}^{+})}^{2}.
\]
This gives the first required inequality. The second required inequality follows from (\ref{eq:est2}).
This completes the proof of tis lemma.
\end{proof}

The following result gives an error estimate
for our nonlinear NLMC solution.

\begin{theorem}
Consider the oversampling domain $K_M^{+}$ obtained by enlarging $K$ by $M$ coarse
cell layers.
Let $u$ be the solution of (\ref{eq:monotonePDE}) and $F^{ms}_1(U_{ms})$ be the solution of (\ref{eq:nlmc3})-(\ref{eq:nlmc4}).
Then we have
\[
\|F_{1}^{ms}(U_{ms})-u\|_{a}\leq CH+C_{1}(M)+C_{2}(M)
\]
where
\begin{align*}
C_{1}(M) & =CC_{1}C_{2}^{-1}(1-C^{-1}C_{1}^{-1}C_{2})^{\frac{M}{2}}M^{\frac{d}{2}}\|U_{ms}\|_{s}, \\
C_{2}(M) & =CC_{\kappa}C_{1}^{2}C_{2}^{-2}(1-C^{-1}C_{1}^{-1}C_{2})^{\frac{M}{2}}M^{\frac{d}{2}}\|U_{ms}\|_{s}.
\end{align*}
Moreover, if $M\sim O\Big(\log(H^{-1})+\log(C_{\kappa})\Big)$ and
$H\leq\cfrac{1}{2}$, then we have
\[
\|F_{1}^{ms}(U_{ms})-u\|_{a}\leq CH.
\]
\end{theorem}

\begin{proof}
We will analyze the error by first separating the error into three parts as follows
\[
\|F_{1}^{ms}(U_{ms})-u\|_{a}\leq\|F_{1}(U_{glo})-u\|_{a}+\|F_{1}(U_{ms})-F_{1}^{ms}(U_{ms})\|_{a}+\|F_{1}(U_{glo})-F_{1}^{ms}(U_{ms})\|_{a}.
\]
By Lemma \ref{lem:gloerror}, we have
\[
\|F_{1}(U_{glo})-u\|_{a}\leq CH\|(I-\Pi)(f \tilde{\kappa}^{-1})\|_{s}
\]
and by Lemma \ref{lem:local_estimate}, we have
\begin{align*}
\|F_{1}(U_{ms})-F_{1}^{ms}(U_{ms})\|_{a}^{2} & \leq\sum_{K}\|F_{1}(U_{ms})-F_{1}^{ms}(U_{ms})\|_{a(K)}^{2}\\
 & \leq(1-C^{-1}C_{1}^{-1}C_{2})^{M}\sum_{K}\|F_{1}(U_{ms})-F_{1}^{ms}(U_{ms})\|_{a(K_{M})}^{2}\\
 & \leq2(1-C^{-1}C_{1}^{-1}C_{2})^{M}\sum_{K}\Big(\|F_{1}(U_{ms})\|_{a(K_{M})}^{2}+\|F_{1}^{ms}(U_{ms})\|_{a(K_{M})}^{2}\Big).
\end{align*}
By Lemma \ref{lem:stablity}, we have
\[
\|F_{1}(U_{ms})\|_{a}^{2}\leq CC_{1}^{2}C_{2}^{-2}\|U_{ms}\|_{s}^{2}
\]
and
\[
\|F_{1}^{ms}(U_{ms})\|_{a(K_{M})}^{2}\leq CC_{1}^{2}C_{2}^{-2}\|U_{ms}\|_{s(K_{M})}^{2}.
\]
Therefore, we obtain
\[
\sum_{K}\Big(\|F_{1}(U_{ms})\|_{a(K_{M})}^{2}+\|F_{1}^{ms}(U_{ms})\|_{a(K_{M})}^{2}\Big)\leq CC_{1}^{2}C_{2}^{-2}M^{d}\|U_{ms}\|_{s}^{2}.
\]

Next, we will estimate the term $\|F_{1}(U_{glo})-F_{1}(U_{ms})\|_{a}$.
By Lemma \ref{lem:coercive} and Assumption 3, we have
\begin{align*}
\|F_{1}(U_{glo})-F_{1}(U_{ms})\|_{a}^{2} & \leq C_{2}^{-1}s(F_{2}(U_{glo})-F_{2}(U_{ms}),F_{1}(U_{glo})-F_{1}(U_{ms}))\\
 & \leq C_{2}^{-1}\|F_{2}(U_{glo})-F_{2}(U_{ms})\|_{s}\|F_{1}(U_{glo})-F_{1}(U_{ms})\|_{s}\\
 & \leq C_{\kappa}C_{2}^{-1}\|F_{2}(U_{glo})-F_{2}(U_{ms})\|_{s}\|F_{1}(U_{glo})-F_{1}(U_{ms})\|_{a}
\end{align*}
and by Lemma \ref{lem:local_estimate}, we have
\begin{align*}
\sum_{K}\|F_{2}(U_{glo})-F_{2}(U_{ms})\|_{s(K)}^{2} & =\sum_{K}\|F_{2}^{loc,K}(U_{ms})-F_{2}(U_{ms})\|_{s(K)}^{2}\\
 & \leq CC_{1}^{2}(1-C^{-1}C_{1}^{-1}C_{2})^{M}\sum_{K}\|F_{1}^{loc,K}(U_{ms})-F_{1}(U_{ms})\|_{a(K_{M}^{+})}^{2}\\
 & \leq CC_{1}^{4}C_{2}^{-2}(1-C^{-1}C_{1}^{-1}C_{2})^{M}M^{d}\|U_{ms}\|_{s}^{2}.
\end{align*}
Therefore, we have
\begin{align*}
\|F_{1}(U_{glo})-F_{1}(U_{ms})\|_{a(K)} & \leq CC_{\kappa}C_{1}^{2}C_{2}^{-2}(1-C^{-1}C_{1}^{-1}C_{2})^{\frac{M}{2}}M^{\frac{d}{2}}\|U_{ms}\|_{s}.
\end{align*}
Combining the above results, we obtain
\[
\|F_{1}^{ms}(U_{ms})-u\|_{a}\leq CH\|(I-\Pi)\frac{f}{\tilde{\kappa}}\|_{s}+C(C_{1}C_{2}^{-1}+C_{\kappa}C_{1}^{2}C_{2}^{-2})(1-C^{-1}C_{1}^{-1}C_{2})^{\frac{M}{2}}M^{\frac{d}{2}}\|U_{ms}\|_{s}.
\]

To show the second part of the theorem,
we notice that
\begin{align*}
\|U_{ms}\|_{s} & \leq\|F_{1}^{ms}(U_{ms})\|_{s}\leq C_{\kappa}\|F_{1}^{ms}(U_{ms})\|_{a}\\
 & \leq C_{\kappa}\Big(\|F_{1}^{ms}(U_{ms})-u\|_{a}+\|u\|_{a}\Big).
\end{align*}
If $M$ is large enough such that
\[
M\geq\frac{2\log\Big(C(C_{1}C_{2}^{-1}+C_{\kappa}C_{1}^{2}C_{2}^{-2})\Big)+2\log(C_{\kappa})+\log(H^{-1})-d\log(M)}{\log\Big((1-C^{-1}C_{1}^{-1}C_{2})^{-1}\Big)},
\]
then we have
\[
CC_{\kappa}(C_{1}C_{2}^{-1}+C_{\kappa}C_{1}^{2}C_{2}^{-2})(1-C^{-1}C_{1}^{-1}C_{2})^{\frac{M}{2}}M^{\frac{d}{2}}\leq H\leq\cfrac{1}{2}
\]
and
\begin{align*}
\|F_{1}^{ms}(U_{ms})-u\|_{a} & \leq CH\|(I-\Pi)( f \tilde{\kappa}^{-1})\|_{s}+C(C_{1}C_{2}^{-1}+C_{\kappa}C_{1}^{2}C_{2}^{-2})(1-C^{-1}C_{1}^{-1}C_{2})^{\frac{M}{2}}M^{\frac{d}{2}}\|U_{ms}\|_{s}\\
 & \leq CH\|(I-\Pi)(f \tilde{\kappa}^{-1})\|_{s}+\frac{1}{2}\|F_{1}^{ms}(U_{ms})-u\|_{a}+H\|u\|_{a}.
\end{align*}
This completes the proof of the theorem.

\end{proof}

\section{Numerical results}
 \label{sec:numresults}

In this section, we present numerical results for the proposed method. 
In our examples, we will use simplified local problems to compute
macroscale parameters. These local computations will involve
machine learning algorithms.
We consider following model problems in fractured and heterogeneous porous media:
\begin{itemize}
\item[] \textit{Test 1}: Nonlinear flow problem (unsaturated flow problem)
\item[] \textit{Test 2}: Nonlinear transport and flow problem (two-phase flow problem)
\end{itemize}

\begin{figure}[h!]
\begin{center}
\includegraphics[width=0.3\linewidth]{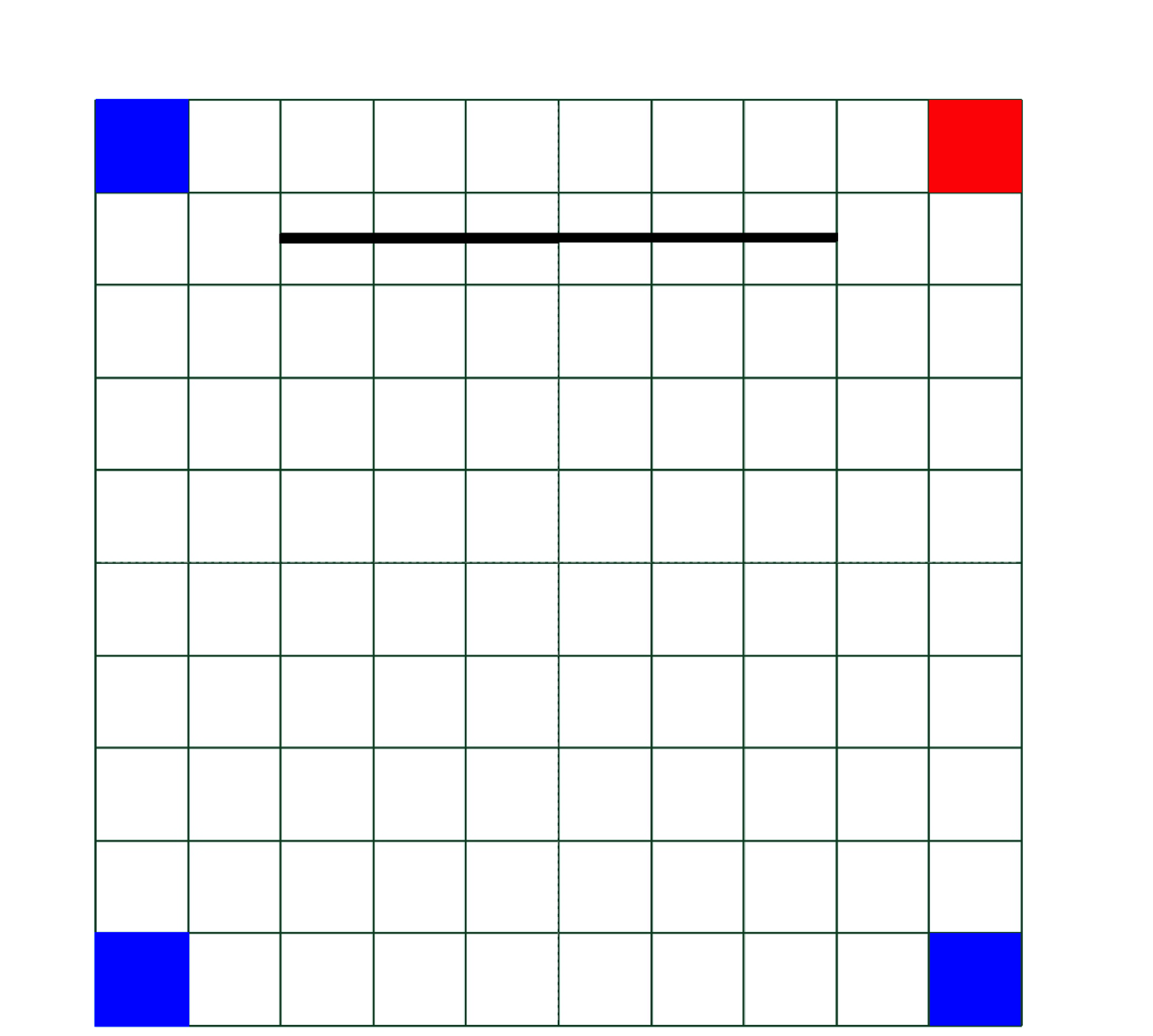}\,\,
\includegraphics[width=0.3\linewidth]{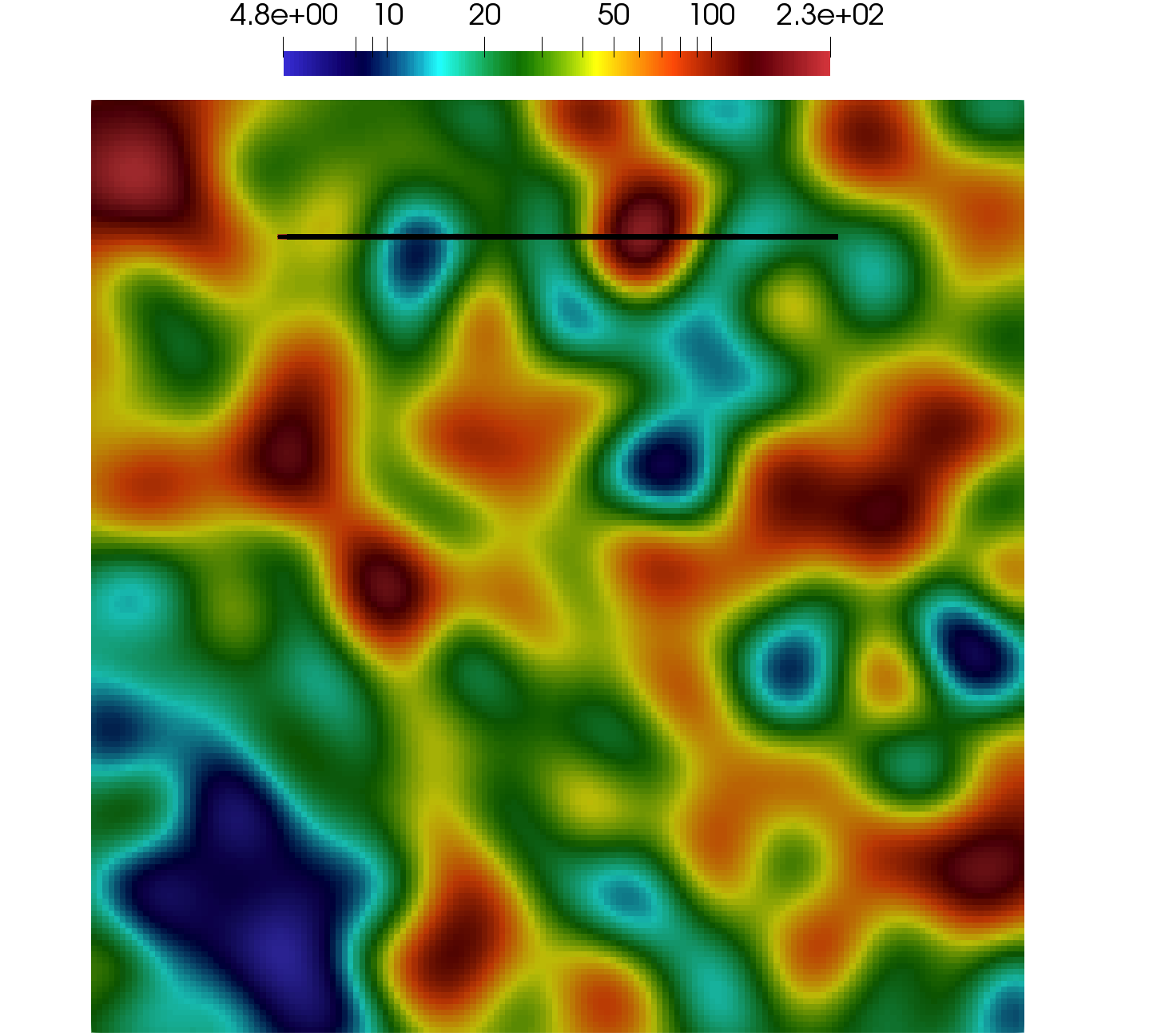}
\end{center}
\caption{Coarse mesh with source term and fracture positions (left).  Heterogeneous porous matrix permeability in $\Omega$ (right)}
\label{fig:kx}
\end{figure}

We solve model problem in 
$\Omega = [0, 1] \times  [0, 1]$ with no flux boundary conditions.  
Heterogeneous porous matrix permeability and location of the source terms and fracture position are depicted in  Figure \ref{fig:kx}. We set source terms $f^{\pm} = \pm q$, $q = 10^3$. We use $10 \times 10$ coarse grid and $160 \times 160$ fine grid. 

\textit{Test 1. }
We consider the solution of the nonlinear equation in fractured heterogeneous media. For the nonlinear coefficients, we use $k^{\alpha \beta}(x, u) = k_s(x) k_r(u)$ with $k_r(u) = \exp(-a |u|)$, $a = 0.1$ ($\alpha, \beta = m,f$).
We set $c^m = 1$, $c^f = 0$, $k_s^f = 10^6$ and $T_{max} = 10^{-3}$ with 20 time steps.

\textit{Test 2. }
We consider the solution of the two-phase flow problem in fractured and heterogeneous porous media. For nonlinear coefficients, we set
$\lambda^w(s)  = s^2$ and $\lambda^n(s)  = (1-s)^2$.
We set $\phi^{\alpha} = 1$ ($\alpha=m,f$), $k^f = 10^3$ and $T_{max} = 6.3 \cdot 10^{-5}$ with 700 time steps. 

\begin{table}[h!]
\begin{center}
\begin{tabular}{ | c | c c c | }
\hline
& MSE & RMSE  (\%)  & MAE (\%) \\
\hline
\multicolumn{4}{|c|}{\textit{Test 1}} \\
\hline
$NN_1$ 	& 0.113 & 3.368 & 2.798 \\
$NN_2$ 	& 0.029 & 1.725 & 1.587 \\
$NN_3$ 	& 0.283 & 5.322 & 4.381 \\
$NN_4$ 	& 0.048 & 2.196 & 2.443 \\
\hline
\end{tabular}\,\,\,\,
\begin{tabular}{ | c | c c c | }
\hline
& MSE & RMSE  (\%)  & MAE (\%) \\
\hline
\multicolumn{4}{|c|}{\textit{Test 2}} \\
\hline
$NN_1$ 	& 0.113 & 3.373 & 1.851 \\
$NN_2$ 	& 0.060 & 2.467 & 1.447 \\
$NN_3$ 	& 0.294 & 5.428 & 2.567 \\
$NN_4$ 	& 0.239 & 4.897 & 2.736 \\
\hline
\end{tabular}
\end{center}
\caption{Learning performance of machine learning algorithm  for \textit{Test 1} and \textit{Test 2} }
\label{tab:ml-t2}
\end{table}

\begin{figure}[h!]
\centering
\includegraphics[width=0.99\linewidth]{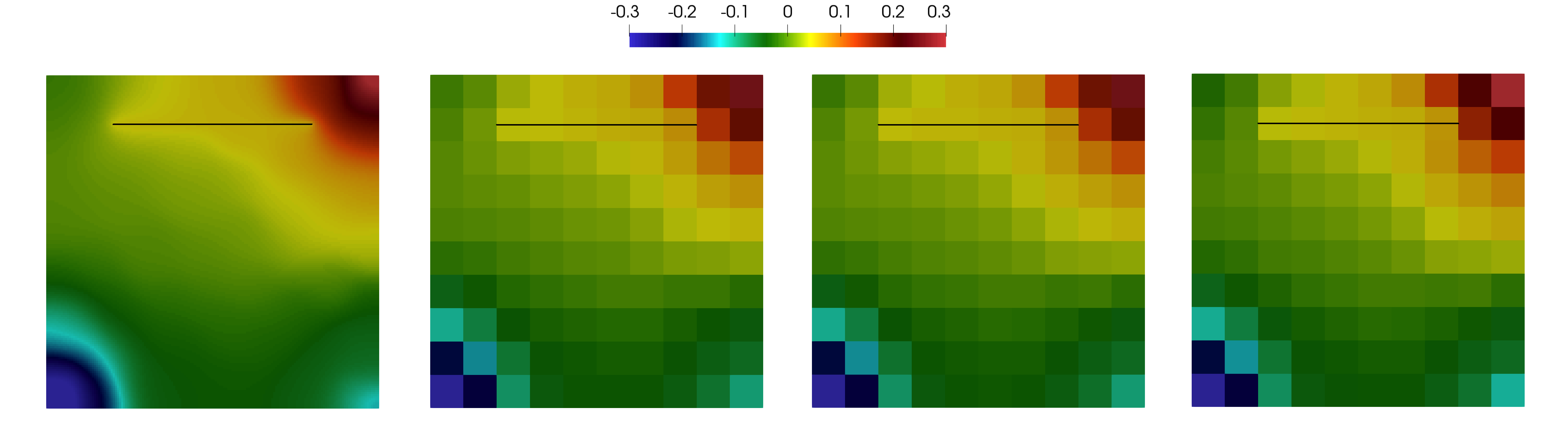}
\caption{Reference fine grid solution ($u^{fine}$), mean value on coarse grid of the fine grid solution ($\overline{u}^{fine}$), coarse grid solution using upscaling method ($\overline{u}^{UP}$)  and coarse grid solution using nonlinear nonlocal machine learning method ($\overline{u}^{NL}$). Nonlinear flow problem (\textit{Test 1}). Pressure on final time $t_m$, $m = 20$ }
\label{fig:uu}
\end{figure}

\begin{figure}[h!]
\centering
\includegraphics[width=0.99\linewidth]{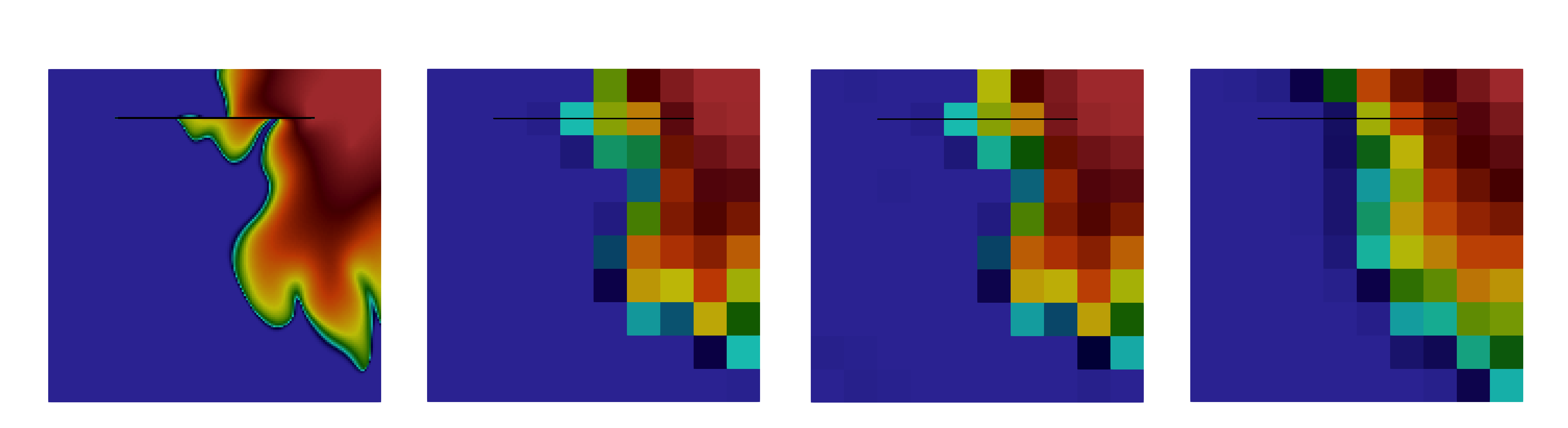}\\
\includegraphics[width=0.99\linewidth]{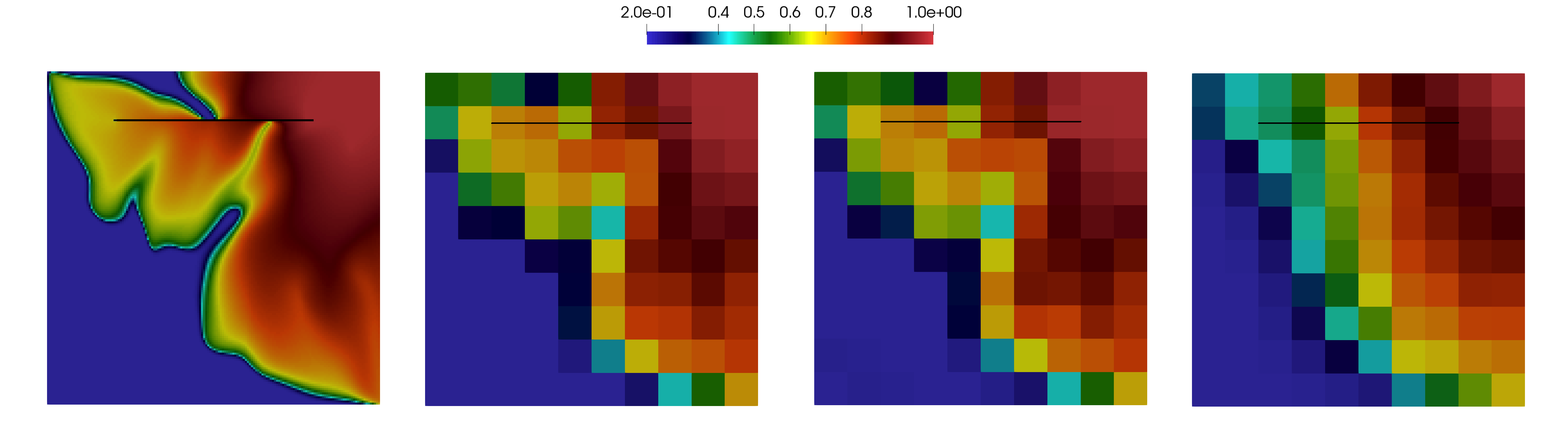}\\
\includegraphics[width=0.99\linewidth]{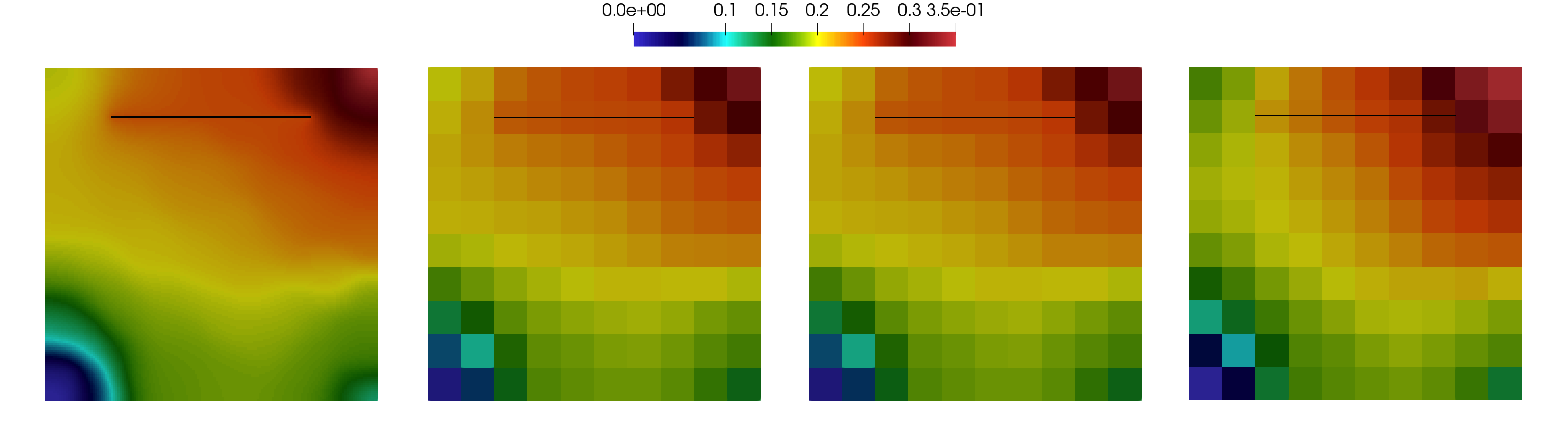}
\caption{Reference fine grid solution ($s^{fine}$, $p^{fine}$), mean value on coarse grid of the fine grid solution ($\overline{s}^{fine}$, $\overline{p}^{fine}$), coarse grid solution using upscaling method ($\overline{s}^{UP}$, $\overline{p}^{UP}$)  and coarse grid solution using nonlinear nonlocal machine learning method ($\overline{s}^{NL}$, $\overline{p}^{NL}$).
Nonlinear flow and transport problem (\textit{Test 2}).
First row: saturation for time $t_m$, $m = 300$.
Second row: saturation for time $t_m$, $m = 700$.
Third row: pressure for time $t_m$, $m = 700$ }
\label{fig:uu-t2}
\end{figure}

Each sample $X_l$ contains the information about heterogeneous permeability and fracture positions up to the fine grid resolution in local domain, coarse grid mean value of the solution in oversampled local domain
\[
\text{\textit{Test 1}:} \, 
X_l = (X_l^k, X_l^f, X_{l+}^{\overline{p}^{m}}),
\quad 
\text{\textit{Test 2}:} \, 
X_l = (X_l^k, X_l^f, X_{l+}^{\overline{p}^{\alpha}}, X_{l+}^{\overline{s}^{\alpha}}, X_{l+}^{\overline{p}^{\beta}}, X_{l+}^{\overline{s}^{\beta}})
\]
and output
\[
\text{\textit{Test 1}:} \, 
Y_l = (T_l^{\alpha \beta, NL} ), \quad \alpha,\beta = m,, 
\quad 
\text{\textit{Test 2}:} \, 
Y_l = (T_l^{\alpha \beta, NL}, T_l^{w, \alpha \beta, NL}), \quad \alpha,\beta = m,f.
\]
Each dataset is divided into training and validation sets with $80:20$ ratio. 

For the training of the neural networks, we use a global dataset, where we extract local information from the fine grid calculations on the global domain $\Omega$.  
We train four neural networks for each type of transmissibility: $NN_1$ for horizontal coarse edges for matrix-matrix flow, $NN_2$ for vertical coarse edges s for matrix-matrix flow, $NN_3$ for matrix - fracture flow and $NN_4$ for fracture - fracture flow.
For calculations, we use $150$ epochs with a batch size $N_b = 90$ and Adam optimizer with learning rate $\epsilon = 0.001$.
For accelerating of the training process of the multi-input CNN, we use GPU.
We use $3 \times 3$ convolutions and $2 \times 2$ maxpooling layers with RELU activation for $X^k$ and $X^f$, and $3 \times 3$ convolutions with RELU activation for $X^{\overline{p}^{m}}$. For each input data, we have 2 layers of CNN with one final fully connected layer. Convolution layer contains  8 and 16 feature maps for $X^k$ and $X^f$; and  4 and 8 feature maps for $X^{\overline{p}^{m}}$. We use dropout with rate 10 \% in each layer in order to prevent over-fitting.
Finally, we combine CNN output and perform two additional fully connected layers with size 200 and 1(one final output).
Presented algorithm is used to learn dependence between multi-input data and upscaled nonlinear transmissibilities.

For error calculation on the dataset, we used  mean square errors, relative mean absolute and relative root mean square errors
\[
MSE = \sum_i |Y_i - \tilde{Y}_i|^2,
\quad
RMSE = \sqrt{ \frac{\sum_i |Y_i - \tilde{Y}_i|^2 }{\sum_i |Y_i|^2 } },
\quad
MAE = \frac{\sum_i |Y_i - \tilde{Y}_i|  }{\sum_i |Y_i|},
\]
where $Y_i$  and $\tilde{Y}_i$ denotes reference and predicted values for sample $X_i$ 
Learning performance for neural networks are presented in Table \ref{tab:ml-t2} 
for \textit{Test 1} and   \textit{Test 2}. 
We observe a good convergence with small error for each neural network.

Next, we consider errors between solution of the coarse grid problem with the
reference  and predicted upscaled transmissibilities.  
To measure difference between reference solution and coarse grid solution, we compute relative $L_2$ error
\[
e(\overline{u}) = \sqrt{
\frac{\sum_{i = 1}^{N^H} (\overline{u}^{fine}_i - \overline{u}_i)^2 }{\sum_{i = 1}^{N^H}  (\overline{u}^{fine}_i)^2 }
},
\]
where $u = p, s$,  $\overline{u}^{fine}$ is the reference solution (mean value on coarse grid of the fine grid solution) and $\overline{u}$ is the solution on the coarse grid. 
In Figure \ref{fig:uu}, we depict solution of the problem for \textit{Test 1} on the fine grid, coarse grid upscaled solution using classic approach  and for new method presented ($u^{fine}$, $\overline{u}^{fine}$, $\overline{u}^{UP}$ and $\overline{u}^{NL}$).  We have $e(\overline{u}^{UP}) = 11.773 \%$  and $e(\overline{u}^{NL}) = 2.155 \%$  at final time.

In Figure \ref{fig:uu-t2}, we depict the solution of the problem  for \textit{Test 2}.  On the first column, we depict a reference fine grid solution ($s^{fine}$, $p^{fine}$), mean value on coarse grid of the fine grid solution ($\overline{s}^{fine}$, $\overline{p}^{fine}$) on the second column, coarse grid solution using upscaling method ($\overline{s}^{UP}$, $\overline{p}^{UP}$) on the third column and coarse grid solution using nonlinear nonlocal machine learning method ($\overline{s}^{NL}$, $\overline{p}^{NL}$) on the fourth column.
On the first, second and third rows, we show a saturation for time $t_m$, $m = 300, 700$ and on fourth row, we have pressure for time $t_m$, $m = 700$.  Fine grid (reference) solution is performed using finite volume approximation with embedded discrete fracture model, where for error calculations, 
we used a mean values of the reference solution on the coarse grid, $\overline{p}^{fine}$ and $\overline{s}^{fine}$.
On the last column of the Figure \ref{fig:uu-t2}, we depict a coarse grid solution using nonlinear nonlocal transmissibilities that calculate based on the machine learning approach.
For machine learning approach, we have
$e(\overline{p}^{NL}) =  0.281 \%$, $ e(\overline{s}^{NL}) = 3.512 \%$, and for upscaling
$e(\overline{p}^{UP}) = 14.063 \%$, $ e(\overline{s}^{UP}) = 13.354 \%$ at final time $t_m$, $m = 700$.

\section{Conclusions}

In the paper, we present a general nonlinear upscaling framework
for
nonlinear differential equations with multiscale coefficients.
The framework is built on nonlinear nonlocal multi-continuum
upscaling concept.
The approach first
identifies test functions for each coarse block, which are used
to identify macroscale variables (called continua).
 In the second stage,
we solve nonlinear
local problems in oversampled regions with some constraints
defined via test functions.
Simplified local problems are proposed for numerical results.
Deep learning algorithms are used to approximate the nonlinear fluxes
that are derived in nonlinear upscaling.
 In the final
stage, macroscale formulation is given and it seeks
the values of
macroscopic variables such that the downscaled field
solves the global problem in a weak sense defined using the test function.
We present an analysis of our approach for an example nonlinear problem.
 We present numerical results for several porous media
applications, including two-phase flow and transport.

\section*{Acknowledgements}

The research of Eric Chung is partially supported by the Hong Kong RGC General Research Fund (Project numbers 14304217 and 14302018)
and CUHK Faculty of Science Direct Grant 2018-19.

\bibliographystyle{plain}
\bibliography{references,references1,references2}

\end{document}